\documentclass[11pt,a4paper,twoside]{amsart}

\usepackage{amsfonts, amsmath,amssymb}
\usepackage{hyperref}
\usepackage[all]{xy}
\usepackage[lite]{amsrefs}
\usepackage{amsthm}
\usepackage{geometry}
\usepackage{graphicx}
\usepackage{wrapfig}
\usepackage{subfig}
\usepackage{appendix}
\usepackage{pstricks, pstricks-add}

\newcommand{\aTop}[2]{\begin{array}{c}{#1}\\{#2}\end{array}}
\newcommand{\Hy}{\mathcal{H}}
\newcommand{\ringO}{\mathcal{O}}

\newcommand{\C}{{\mathbb{C}}}

\newcommand{\Z}{{\mathbb{Z}}}

\newcommand{\N}{{\mathbb{N}}}
\newcommand{\rationals}{{\mathbb{Q}}}

\newcommand*{\Homol}{\operatorname{H}}

\newcommand{\PSO}{\mathrm{PSO}}

\renewcommand{\leq}{\leqslant}

\newcommand{\Afour}{\mathcal{A}_4}
\newcommand{\Sthree}{\mathcal{D}_3}
\newcommand{\Kleinfourgroup}{\mathcal{D}_2}
\newcommand{\Betti}{\Z^{\beta_1}}

\theoremstyle{plain}
\newtheorem{thm}{\bfseries Theorem}
\newtheorem{theorem}[thm]{\bfseries Theorem}
\newtheorem{Lem}[thm]{\bfseries Lemma}
\newtheorem{lemma}[thm]{\bfseries Lemma}

\newtheorem{proposition}[thm]{\bfseries Proposition}

\newtheorem{corollary}[thm]{\bfseries Corollary}

\newtheorem{df}[thm]{\bfseries Definition}

\theoremstyle{remark}

\newtheorem{remark}[thm]{\bfseries Remark}

\newtheorem*{ConditionA}{\bfseries Condition A}
\newtheorem*{ConditionB}{\bfseries Condition B}
\newcommand{\cellCondition}{A}
\newcommand{\weakerCondition}{B}

\newcommand{\circlegraph}{ %reduced 2-torsion subgraph for m = 7, 15, 39
\begin{pspicture}     (-0.21,-0.155)(0.21,0.155) 
             \pscircle(0,0.0){0.15}
             \psdots(0.15,0.0)
\end{pspicture} }
 
\newcommand{\edgegraph}{ %reduced 2-torsion subgraph for m = 11, 19, 43, 67, 139, 163
\begin{pspicture}(-0.3,-0.1)(0.3,0.3)
\psdots(-0.2,0.0)
\psdots(0.2,0.0)
\psline(-0.2,0.0)(0.2,0.0)
\end{pspicture} }

\newcommand{\graphFive}{  % reduced 2-torsion subgraph for m = 5, 10, 13
\begin{pspicture}(-0.21,-0.155)(0.21,0.155)
\pscircle(0,0.0){0.15}
\psdots(-0.15,0)
\psdots(0.15,0)
\psline(-0.15,0)(0.15,0)
\end{pspicture} }

\newcommand{\graphTwo}{  % reduced 2-torsion subgraph for m = 2
\begin{pspicture}(0,0.05)(0.8,0.4)
\pscircle(0.2,0.2){0.15}
\psdots(0.35,0.2)
\psline(0.35,0.2)(0.65,0.2)
\psdots(0.65,0.2)
\end{pspicture} }

\begin{document}

\title[On the equivariant $K$-homology of PSL$_2$ of the imaginary quadratic integers]{On the equivariant $K$-homology \\ of PSL$_2$ of the imaginary quadratic integers}
\author{Alexander D. Rahm}
\date{\today}

\begin{abstract}
 We establish formulae for the part due to torsion of the equivariant $K$-homology of all the Bianchi groups 
 (PSL$_2$ of the imaginary quadratic integers), 
 in terms of elementary number-theoretic quantities.
To achieve this, we introduce a novel technique in the computation of Bredon homology: \emph{representation ring splitting}, 
which allows us to adapt the recent technique of torsion subcomplex reduction from group homology to Bredon homology.
\end{abstract}

\maketitle

\section*{Introduction}
Let $\ringO_{-m}$ be the ring of algebraic integers in the imaginary quadratic number field $\rationals(\sqrt{-m)}\thinspace)$;
the \textit{Bianchi groups} are the projective special linear groups PSL$_2(\ringO_{-m})$.
In this paper, we establish formulae for the part due to torsion of the Bredon homology of the Bianchi groups,
with respect to the family $\mathfrak{Fin}$ of finite subgroups and coefficients in the complex representation ring $R_\C$,
from which we deduce their equivariant $K$-homology.
Then we use the fact that 
the Baum--Connes assembly map from the equivariant $K$-homology to the 
$K$-theory of the reduced $C^*$-algebras of the Bianchi groups is an isomorphism.
This has been inspired by works of Sanchez-Garcia \cites{Sanchez-Garcia, Sanchez-Garcia_Coxeter},
and allows us to obtain the isomorphism type of the latter operator $K$-theory, which would be extremely hard to compute from the reduced $C^*$-algebras. 
Case-by-case computations on the machine have already been carried out for the equivariant $K$-homology of the Bianchi groups~\cite{noteAuxCRAS}
(see~\cite{Fuchs} for a way to extend them to non-trivial class group cases, alternative to the current implementation by the author),
but by their nature, they can of course only cover a small finite collection of Bianchi groups. 
In order to obtain the desired formulae for all Bianchi groups, we set up an adaptation to Bredon homology of \emph{torsion subcomplex reduction}, 
a technique which has recently been formulated for group homology~\cite{AccessingFarrell},
and some elements of which had already been used earlier on as ad hoc tricks by Soul\'e~\cite{Soule}.
A priori, it is possible with our methods to treat any discrete group with a nice action on a cell complex;
and another class of examples for this is work in advanced progress jointly with Lafont, Ortiz and Sanchez-Garcia~\cite{LORS},
namely hyperbolic Coxeter groups, many of which are not arithmetic.
Please note that definitions of Bredon homology and equivariant $K$-homology are given in~\cite{MislinValette}, 
\cite{Sanchez-Garcia} and~\cite{Sanchez-Garcia_Coxeter}, 
so we will not recall them in this paper.

\section{Statement of the results}
Denote by $\underbar{E}\Gamma$ the classifying space for proper actions of $\Gamma$,
and denote by $\underbar{B}\Gamma := \Gamma\backslash\underbar{E}\Gamma$ the orbit space.
\begin{theorem} \label{splitting}
Let $\Gamma$ be a Bianchi group or one of their subgroups.
 Then the Bredon homology $\Homol^\mathfrak{Fin}_n(\Gamma; \thinspace R_\C)$ 
 is concentrated in degrees $n \in \{0, 1, 2\}$ and splits as a direct sum over 
\begin{enumerate}
 \item the orbit space homology $\Homol_n(\underbar{\rm B}\Gamma; \thinspace \Z)$, 
 \item a submodule $\Homol_n(\Psi_\bullet^{(2)})$ determined by a reduced $2$--torsion subcomplex for $(\underline{\rm E}\Gamma, \Gamma)$
 \item and a submodule $\Homol_n(\Psi_\bullet^{(3)})$ determined by a reduced $3$--torsion subcomplex for $(\underline{\rm E}\Gamma, \Gamma)$.
\end{enumerate}
\end{theorem}
These submodules are given as follows, except for PSL$_2$ over the Gaussian and Eisensteinian integers.
The additional units in these two rings induce some particularities which we would like to avoid.
This does not cause any harm, because the Bredon homology and equivariant $K$-homology of these two Bianchi groups have already been computed~\cite{noteAuxCRAS}.
So for the remainder of this article, the term ``Bianchi group'' will stand for PSL$_2$ over a ring of imaginary quadratic integers excluding these two.

\begin{theorem} \label{2}
 The part due to $2$-torsion of the Bredon complex of a Bianchi group $\Gamma$ has homology
 \begin{center}
  $\Homol_n(\Psi_\bullet^{(2)}) \cong \begin{cases}
                                      \Z^{z_2}\oplus (\Z/2)^\frac{d_2}{2},& n = 0,\\
                                      \Z^{o_2},& n = 1,\\
                                      0,&\text{\rm otherwise},
                                     \end{cases}
  $
 \end{center}
where $z_2$ counts the number of conjugacy classes of subgroups of type $\Z/2$ in $\Gamma$,
$o_2$ counts the conjugacy classes of those of them which are not contained in any $2$-dihedral subgroup,
and $d_2$ counts the number of $2$-dihedral subgroups, whether or not they are contained in a tetrahedral subgroup of $\Gamma$.
\end{theorem}

\begin{theorem} \label{3}
 The part due to $3$-torsion of the Bredon complex of a Bianchi group $\Gamma$ has homology
 \begin{center}
  $\Homol_n(\Psi_\bullet^{(3)}) \cong \begin{cases}
                                      \Z^{2o_3+\iota_3},& n = 0 \medspace \text{\rm or }1,\\
                                      0,&\text{\rm otherwise},
                                     \end{cases}
  $
 \end{center}
where amongst the subgroups of type $\Z/3$ in $\Gamma$,
$o_3$ counts the number of conjugacy classes of those of them which are not contained in any $3$-dihedral subgroup,
and $\iota_3$ counts the conjugacy classes of those of them which are contained in some $3$-dihedral subgroup in $\Gamma$.
\end{theorem}

Note that there are formulae for the numbers $o_2, z_2, d_2, o_3$ and $\iota_3$ 
in terms of elementary number-theoretic quantities~\cite{KraemerDiplom}, 
which are very easy to evaluate on the machine~\cite{AccessingFarrell}*{appendix}.

Together with Theorem~\ref{Bredon_to_K-homology} below, we obtain the following formulae for the equivariant $K$-homology of the Bianchi groups.
Note for this purpose that for a Bianchi group $\Gamma$, 
there is a model for \underline{E}$\Gamma$ of dimension 2,
so $\Homol_2(\underline{\rm B}\Gamma ; \thinspace \Z) \cong \Z^{\beta_2}$ is torsion-free.
It has been shown in~\cite{Serre} 
that the naive Euler characteristic of the Bianchi groups vanishes 
(again excluding the two special cases of Gaussian and Eisensteinian integers),
for $\beta_i = \dim \Homol_i(\underline{\rm B}\Gamma ; \thinspace \rationals)$ we have
$\beta_0 -\beta_1 +\beta_2 = 0$  and $\beta_0 = 1$.
\begin{corollary}
 For any Bianchi group $\Gamma$, the short exact sequence of Theorem~\ref{Bredon_to_K-homology} splits into
$K^\Gamma_0(\underbar{\rm E}\Gamma) \cong \Z \oplus \Z^{\beta_2}  \oplus \Z^{z_2} \oplus (\Z/2)^\frac{d_2}{2} \oplus \Z^{2o_3+\iota_3}$. 
Furthermore, $K^\Gamma_1(\underbar{\rm E}\Gamma) \cong \Homol_1(\underline{\rm B}\Gamma; \thinspace \Z) \oplus \Z^{o_2} \oplus \Z^{2o_3+\iota_3}$. 
\end{corollary}
A table evaluating these formulas for a range of Bianchi groups is given in Appendix \ref{the appendix}.
That table agrees with the machine calculations which were carried out by the author
for all cases of class number $1$ and $2$ with~\cite{BianchiGP}
in the way described in the author's PhD thesis~\cite{RahmThesis} 
(only the cases of class number $1$ were covered at that time), 
following the method of \cites{Sanchez-Garcia, Sanchez-Garcia_Coxeter}.

The remainder of the equivariant $K$-homology of $\Gamma$ is given by 2-periodicity. \\
As the Baum-Connes conjecture is verified by the Bianchi groups, 
these equivariant $K$-homology groups are isomorphic to the $K$-theory of the reduced group $C^*$-algebras.

\subsubsection*{Organisation of the paper}
In Section~\ref{the cell complex}, we describe the cell complex on which we study the Bianchi groups through their action.
In Section~\ref{recalls}, we make recalls about the Bredon chain complex;
in Section~\ref{TSR}, we recall torsion subcomplexes and their reduction.
In Section~\ref{Representation ring splitting}, we give the proof of Theorem \ref{splitting}.
In Section~\ref{Reduction of the torsion subcomplexes}, we give the proof of Theorems~\ref{2} and~\ref{3}.
Section~\ref{general} is a study on how the methods of this paper can be generalised,
and establishes a statement for arbitrary classifying spaces for proper actions with zero-dimensional singular part.
We summarize information on the assembly map for the Bianchi groups in Section~\ref{The assembly map}.
Section~\ref{the appendix} appends results from machine computations on the Bredon homology of the Bianchi groups.

\subsubsection*{Acknowledgements} The author would like to thank \mbox{Conchita Mart\'inez P\'erez}, Alain Valette
and especially Rub\'en S\'anchez-Garc\'\i{}a  for helpful discussions, and an anonymous referee for helpful suggestions.

\section{A suitable model for the classifying space for proper actions} \label{the cell complex}
In order to construct a suitable model for $\underbar{\rm E}\Gamma$
when $\Gamma$ is a Bianchi group, we start with the natural action of $\Gamma < $ SL$_2(\C)$
on its associated symmetric space SL$_2(\C)/_{\text{SU}_2}$.
The latter space is isomorphic to hyperbolic $3$-space $\Hy^3$,
and under this isomorphism, the action can be expressed by the M\"obius transformation formula
when embedding $\Hy^3$ into the quaternions for carrying out the occurring division.
This action makes $\Hy^3$ into a model for $\underbar{\rm E}\Gamma$,
but not a cocompact one; a fact which is annoying for Bredon homology computations.
In order to reach the desired cocompactness,
one could for instance apply the Borel--Serre compactification;
but for explicit homological computations,
it is more convenient to construct a $2$-dimensional $\Gamma$-equivariant retract of $\Hy^3$.
A very conceptual way for such a construction is to use the reduction theory of Borel and Harish-Chandra.
This has been done by Harder~\cite{Harder}, 
put into practice for the Bianchi groups by his student Mendoza~\cite{Mendoza}
and implemented on the machine by Vogtmann~\cite{Vogtmann}.
For a reader desiring to extend the present investigations to other arithmetic groups,
this construction should be the method of choice.
However, the above mentioned computer implementation has not been preserved over time,
and the author has implemented a different model for $\underbar{\rm E}\Gamma$
in order to produce the machine results presented in the appendix (Section~\ref{the appendix}).
That model, due to Fl\"oge~\cite{Floege}, 
comes with some peculiarities as one does first adjoin the singular cusps of the $\Gamma$-action to $\Hy^3$
before retracting $\Gamma$-equivariantly (see details in~\cite{RahmFuchs}).
In order to keep the cell stabilisers finite, one has therefore to start, instead of with $\Hy^3$,
with its Borel--Serre bordification. Then at the cusps, one gets $2$-tori instead of points stabilised by fundamental groups of $2$-tori.
A variation of this has been pursued in detail by Fuchs~\cite{Fuchs}.

On the other hand, as Fl\"oge's model is obtained from the boundary of Bianchi's fundamental polyhedron,
it allows to study the geometry of the Bianchi groups using the insights of old masters like Luigi~Bianchi~\cite{Bianchi1892},
Felix~Klein~\cite{binaereFormenMathAnn9} and Henri Poincar\'e~\cite{Poincare}.
We note that the Harder--Mendoza model and the Fl\"oge model coincide when the ring of integers is a principal ideal domain,
because then there are no singular cusps.
The reader may choose to think of the $\Gamma$-equivariant retract $X$ of $\Hy^3$
in this paper either as the Harder--Mendoza model or as the Fl\"oge model,
according to her or his research interests.
What is important in any case, is to provide $X$ with a $\Gamma$-invariant 
cell structure in which $\Gamma$
does not perform any ``inversions'' of cells
(mapping a cell to itself without restricting to the identity map on that cell).
For theoretical purposes, this can always be achieved using the barycentric subdivision
of a given cell structure; but for explicit computations, 
one should instead make a subdivision along the symmetries of the cells.
Such a subdivision of $X$ has been achieved on the machine~\cite{Higher_torsion},
and has been used for the appendix (Section~\ref{the appendix}).
As a consequence of any subdivision which eliminates the ``inversions'' of cells,
\begin{itemize}
 \item the $2$-cells of $X$ are trivially stabilised;
 \item any $1$-cell stabiliser is of isomorphism type $\Z/_{n\Z}$, $n \in \{1, 2, 3\}$,
 and performs a rotation of order $n$ on $\Hy^3$ such that its $1$-cell is on the rotation axis.
\end{itemize}
On the $0$-cells, obviously any stabilising map restricts to the identity map,
so the cell stabiliser can be any of the finite subgroups of $\Gamma$.

\section{Recalls about the Bredon chain complex} \label{recalls}
As described in~\cite{MislinValette}, 
the equivariant $K$-homology of the classifying space for proper actions $K^G_*(\underbar{\rm E}G)$
can be computed by means of the Bredon homology with coefficients in the complex representation ring,
$\Homol^\mathfrak{Fin}_*(G; R_\C)$. 
More precisely, when we have a classifying space for proper actions of dimension at most 2, 
then the Atiyah---Hirzebruch spectral sequence from its Bredon homology to its equivariant \mbox{$K$-homology}
degenerates on the $E^2$-page and directly yields the following.
\begin{thm}[\cite{MislinValette}] \label{Bredon_to_K-homology}
 Let $G$ be an arbitrary group such that $\dim \underbar{\rm E}G \leq 2$. 
Then there is a natural short exact sequence
$$0 \to \Homol^\mathfrak{Fin}_0(G; R_\C) \to K^G_0(\underbar{\rm E}G) \to \Homol^\mathfrak{Fin}_2(G; R_\C) \to 0 $$
and a natural isomorphism $\Homol^\mathfrak{Fin}_1(G; R_\C) \cong K^G_1(\underbar{\rm E}G)$.
\end{thm}

We will follow S\'anchez-Garc\'ia's treatment \cites{Sanchez-Garcia,Sanchez-Garcia_Coxeter}. \\
Consider the Bianchi group \mbox{$\Gamma := \mathrm{PSL_2}(\mathcal{O}_{-m})$.} 
We use the $2$-dimensional model $X$ for $\underbar{E}\Gamma$ described in Section~\ref{the cell complex}.

Denote by $\Gamma_\sigma$ the stabiliser of a cell $\sigma$, and by $R_\C(G)$ the complex representation ring of a group $G$.
 We will study the Bredon chain complex
$$\xymatrix{
0 \ar[r] & \bigoplus\limits_{\sigma \in \thinspace_\Gamma \backslash X^{(2)}} R_\C (\Gamma_\sigma) \ar[r]^{\Psi_2 } & 
\bigoplus\limits_{\sigma \in \thinspace_\Gamma \backslash X^{(1)}} R_\C (\Gamma_\sigma)  \ar[r]^{\Psi_1} &
\bigoplus\limits_{\sigma \in \thinspace_\Gamma \backslash X^{(0)}} R_\C (\Gamma_\sigma) \ar[r] & 0,
}  $$
of our $\Gamma$-cell complex $X$. 
As $X$ is a model for the classifying space for proper $\Gamma$-actions, 
the homology of this Bredon chain complex is the Bredon homology $\Homol^\mathfrak{Fin}_p(\Gamma; \thinspace R_\C)$
of~$\Gamma$~\cite{Sanchez-Garcia}.

\section{Recalls about torsion subcomplexes} \label{TSR}

In this section we recall the $\ell$--torsion subcomplexes theory of~\cite{AccessingFarrell}
for the calculation of group homology, which we are going to adapt to the calculation of Bredon homology in this paper. 
 We require any discrete group $\Gamma$ 
under our study to be provided with a $\varGamma$\textit{--cell complex}, 
that is a  finite-dimensional cell complex $X$ with cellular 
$\Gamma$--action such that each cell stabilizer fixes its cell point-wise.  
Let $\ell$ be a prime number.
\begin{df}
 The \emph{$\ell$--torsion subcomplex} of a $\Gamma$--cell complex $X$
 consists of all the cells of $X$ whose stabilizers in~$\Gamma$ 
 contain elements of order $\ell$.
\end{df}

We further require that the fixed point set~$X^G$ 
be acyclic for every nontrivial finite $\ell$--subgroup $G$ of~$\Gamma$.
Then Brown's proposition X.(7.2)~\cite{Brown} specializes as follows. 
\begin{proposition} \label{Brownian}
 There is an isomorphism between the $\ell$--primary parts of the Farrell cohomology of~$\Gamma$
 and the
 $\Gamma$--equivariant Farrell cohomology of the $\ell$--torsion subcomplex.
\end{proposition}
With the splitting of Theorem \ref{splitting}, we obtain a Bredon homology analogue of the above proposition,
finding for each occurring prime $\ell$ a component of the Bredon homology carried by the $\ell$--torsion subcomplex.

For a given Bianchi group, the $\ell$--torsion subcomplex can be quite large.  
It turns out to be useful to reduce this subcomplex, and we identify two conditions under 
which we can do this in a way that Proposition~\ref{Brownian} still holds.
 
\begin{ConditionA} \label{cell condition}
In the $\ell$--torsion subcomplex, let $\sigma$ be a cell of dimension $n-1$
which lies in the boundary of precisely two $n$--cells 
representing different orbits, $\tau_1$ and~$\tau_2$.   
Assume further that no higher-dimensional cells of the $\ell$--torsion subcomplex touch $\sigma$;
and that the $n$--cell stabilizers admit an isomorphism
$\Gamma_{\tau_1} \cong \Gamma_{\tau_2}$. 
\end{ConditionA}

\begin{ConditionB} 
The inclusion of the cell stabilizer $\Gamma_\sigma$ into $\Gamma_{\tau_1}$ and
$\Gamma_{\tau_2}$ induces isomorphisms on$\mod \ell$ cohomology.
\end{ConditionB}

When both conditions are satisfied in the $\ell$--torsion subcomplex,  
we merge the cells $\tau_1$ and $\tau_2$ along~$\sigma$ and 
do so for their entire orbits.  
The effect of this merging is to decrease the size of the  
$\ell$--torsion subcomplex without changing its
$\Gamma$--equivariant Farrell cohomology.  
This process can often be repeated: by a ``terminal vertex,'' we will denote a vertex with 
no adjacent higher-dimensional cells and precisely one adjacent edge in the quotient space,  
and by ``cutting off'' the latter edge,
we will mean that we remove the edge together with the terminal vertex from our cell complex.
\begin{df}
 The \emph{reduced $\ell$--torsion subcomplex} associated to a $\Gamma$--cell complex~$X$
 is the cell complex obtained by recursively merging orbit-wise all the pairs of cells satisfying 
 conditions~$\cellCondition$ and~$\weakerCondition$,
 and cutting off edges that admit a terminal vertex when condition~$\weakerCondition$ 
 is satisfied.
\end{df}
The following theorem, stating that Proposition~\ref{Brownian} 
still holds after reducing,
is proved in~\cite{AccessingFarrell}:
\textit{
 There is an isomorphism between the $\ell$--primary parts of the Farrell cohomology of~$\Gamma$ and the
 $\Gamma$--equivariant Farrell cohomology of the reduced $\ell$--torsion subcomplex.
}
In the case of a trivial kernel of the action on the $\Gamma$--cell complex, 
this allows one to establish general formulae for the Farrell cohomology of~$\Gamma$ 
\cite{AccessingFarrell}.
Analogously, we will use our adaptation of torsion subcomplex reduction to Bredon homology in order to prove the formulae in Theorems~\ref{2} and~\ref{3}.

\begin{table} 
 \begin{center} 
 \caption{Connected components of reduced torsion subcomplex quotients for the Bianchi groups} \label{table:subcomplexes}
 \label{one}
 \begin{tabular}{|c|c|c|c|c|}
 \hline & & &&\\
\begin{tabular}{c}$2$--torsion\\subcomplex components\end{tabular}
& counted by & & 
\begin{tabular}{c}$3$--torsion\\subcomplex components\end{tabular} 
& counted by
\\  \hline & & &&\\
 $\circlegraph \thinspace \Z/2$ & $o_2$ & &  $\circlegraph \thinspace \Z/3$ & $o_3$\\
 & & & &\\
$\Afour \edgegraph \Afour$ & $\iota_2$ & & $\Sthree \edgegraph \Sthree$ & $\iota_3$\\
 & & &&\\
$\Kleinfourgroup \graphFive \thinspace \Kleinfourgroup$ & $\theta$&&&\\
 & & &&\\
$\Kleinfourgroup \graphTwo \Afour$ & $\rho$ &&&\\
\hline
\end{tabular}
\end{center}
\end{table}

Table \ref{one} displays the types of connected components of reduced torsion subcomplex quotients for the action of the Bianchi groups on hyperbolic space:
Norbert Kr\"amer~\cite{Kraemer}*{Satz 8.3 and Satz 8.4} has shown that the types $\circlegraph$, $\graphFive$, $\graphTwo$ and $\edgegraph$ are all possible homeomorphism types which 
connected components of the $2$--torsion subcomplex of the action of a Bianchi group on hyperbolic space can have.
For $3$--torsion, the existence of only the two specified types was already proven in~\cite{Rahm_homological_torsion}.

\section{Representation ring splitting} \label{Representation ring splitting}

Recall the following classification of Felix Klein~\cite{binaereFormenMathAnn9}.
\begin{Lem}[Klein] \label{finiteSubgroups}
The finite subgroups in $\mathrm{PSL}_2(\ringO)$
 are exclusively of isomorphism types the cyclic groups of orders one, two and three, the $2$-dihedral group
\mbox{$\Kleinfourgroup \cong \Z/2 \times \Z/2$}, the $3$-dihedral group $\Sthree$ or the tetrahedral group isomorphic to the alternating group~$\Afour$.
\end{Lem}
We further use the existence of geometric models for the Bianchi groups in which all edge stabilisers are finite cyclic
and all cells of dimension $2$ and higher are trivially stabilised.
Therefore, the system of finite subgroups of the Bianchi groups admits inclusions only emanating from cyclic groups. 
This makes the Bianchi groups and their subgroups subjects to the splitting of Bredon homology stated in Theorem~\ref{splitting}.

The proof of Theorem~\ref{splitting} is based on the above particularities of the Bianchi groups,
and applies the following splitting lemma for the involved representation rings
 to a Bredon complex for~$(\underline{\rm E}\Gamma, \Gamma)$.
\begin{lemma} \label{representation ring splitting}
Consider a group $\Gamma$ such that every one of its finite subgroups is either cyclic of order at most~$3$, or of one of the types 
$\Kleinfourgroup, \Sthree$ or~$\Afour$.
Then there exist bases of the complex representation rings of the finite subgroups of~$\Gamma$,
 such that simultaneously every morphism of representation rings
 induced by inclusion of cyclic groups into finite subgroups of~$\Gamma$,
 splits as a matrix into the following diagonal blocks.
\begin{enumerate}
 \item A block of rank $1$ induced by the trivial and regular representations,
 \item a block induced by the $2$--torsion subgroups
 \item and a block induced by the $3$--torsion subgroups.
\end{enumerate}
\end{lemma}
As this splitting holds simultaneously for every morphism of representation rings,
we have such a splitting for every morphism of formal sums of representation rings,
and hence for the differential maps of the Bredon complex for any Bianchi group and any of their subgroups.
\begin{proof}
We consider the complex representation ring of these finite groups as the free 
$\Z$-module the basis of which are the irreducible characters of the group
 (here and in the following, we identify representations with their associated characters,
 basing ourselves on~\cite{Serre_linear}).
 The trivial group admits only the representations by the identity matrix,
 so its only irreducible character is given by the trace $1$.
For the other finite subgroups of the Bianchi groups,
 we make the below choices of conjugacy representatives of their elements,
 and specify below the irreducible characters by the values they take on these elements.
For  $\Z/2 = \langle\, g \, | \, g^2 = 1 \rangle$, 
we transform the tables of irreducible characters into the following basis for the representation ring.
$$
\left(
\begin{array}{c|rr}
       \Z/2 & 1 & g \\
       \hline
       \rho_1 & 1 & 1 \\
       \rho_2 & 1 & -1 \\
\end{array}
\right)
\mapsto
\left(
\begin{array}{c|rr}
       \Z/2 & 1 & g \\
       \hline
       \rho_1 +\rho_2 & 2 & 0 \\
	       \rho_2 & 1 & -1 \\
\end{array}
\right).
$$
 Let $j = e^\frac{2\pi i}{3}$.  Then for $\Z/3 = \langle\, h \, |\, h^3 = 1\, \rangle$, we transform
$$
\left(
\begin{array}{c|ccc}
       \Z/3 & 1 & h & h^2\\
       \hline
       \sigma_1 & 1 & 1  & 1\\
       \sigma_2 & 1 & j & j^2\\
       \sigma_3 & 1 & j^2 & j\\
\end{array}
\right)
\mapsto
\left(
\begin{array}{c|ccc}
       \Z/3 & 1 & h & h^2\\
       \hline
\sum_i \sigma_i & 3 & 0  & 0\\
       \sigma_2 & 1 & j & j^2\\
       \sigma_3 & 1 & j^2 & j\\
\end{array}
\right),
$$
and for the the Klein four-group 
\mbox{$\Kleinfourgroup = \langle\, a, \, b \, |\, a^2 = b^2 = (ab)^2 = 1\, \rangle$},
$$
\left(
\begin{array}{c|rrrr}
       \Kleinfourgroup & 1 & a & b & ab\\
       \hline
       \xi_1 & 1 & 1 & 1 & 1\\
       \xi_2 & 1 & -1 & -1 & 1\\
       \xi_3 & 1 & -1 & 1 & -1 \\
       \xi_4 & 1 & 1 & -1 & -1\\
\end{array}
\right)
\mapsto
\left(
\begin{array}{c|rrrr}
       \Kleinfourgroup & 1 & a & b & ab\\
       \hline
	      \xi_1 & 1 &  1 & 1  & 1\\
       \xi_2 -\xi_1 & 0 & -2 & -2 & 0\\
       \xi_3 -\xi_1 & 0 & -2 & 0  & -2 \\
       \xi_4 -\xi_1 & 0 &  0 & -2 & -2\\
\end{array}
\right).
$$
In all these three cases, we encounter either only 2--torsion or only 3--torsion.
In the cases of $\Sthree$ and $\Afour$, where we have both types of torsion,
 we split the representation ring into a direct sum of submodules associated to
 respectively the trivial subgroup, 2--torsion and 3--torsion.
We achieve this in the following way.
\begin{itemize}
\item We write the symmetric group $\Sthree = \langle\, (12), (123) \,|$ cycle relations $
\rangle$ in cycle type notation. Then we apply the following base transformation to its character table.
\begin{center}
$
\left(
\begin{array}{c|rrr}
       \Sthree & 1 & (12) & (123) \\
       \hline
       \pi_1 & 1 & 1 & 1  \\
       \pi_2 & 1 & -1 & 1 \\
       \pi_3 & 2 & 0 & -1 \\
\end{array} \right) \mapsto  \left(
\begin{array}{c|rrr}
             \Sthree & 1 & (12) & (123) \\
       \hline
				     \pi_1 & 1 & 1 & 1  \\
	\widetilde{\pi_2}:=  \pi_2 - \pi_1 & 0 & -2 & 0 \\
   \widetilde{\pi_3}:= \pi_3 -\pi_2 -\pi_1 & 0 & 0 & -3 \\
\end{array} \right)$
\end{center}
\item We also write the alternating group $\Afour$ in cycle type notation,
 and let again \mbox{$j = e^\frac{2\pi i}{3}$.} Then we transform 
$$
\left(
\begin{array}{c|rrrr}
       \Afour & 1 & (12)(34) & (123) & (132)\\
       \hline
       \chi_1 & 1 &  1 & 1   & 1\\
	\chi_2& 3 & -1 &0    & 0 \\
       \chi_3 & 1 &  1 & j   & j^2\\
       \chi_4 & 1 &  1 & j^2 & j
\end{array}
\right)
\mapsto 
\left(
\begin{array}{c|rrrr}
       \Afour & 1 & (12)(34) & (123) & (132)\\
       \hline
			 \chi_1 & 1 & 1  & 1   & 1\\
\chi_2 - \chi_1 -\chi_3 -\chi_4 & 0 & -4 & 0   & 0 \\
	         \chi_3 -\chi_1 & 0 &  0 & j-1 & j^2 -1\\
		 \chi_4 -\chi_1 & 0 &  0 & j^2-1 & j-1
\end{array}
\right).
$$
\end{itemize}
The above transformed tables consist no more only of irreducible characters,
 but clearly, they still are bases for the complex representation rings of the concerned groups.

For an injective morphism $H \hookrightarrow G$ of finite groups,
 we compute as follows the map induced on the complex representation rings.
 We restrict the characters $\phi_i$ of $G$ to the image of $H$, and write $\phi_i\downarrow$ for the restricted character.
 Then  we consider the scalar products 
$$(\phi_i\downarrow | \tau_j) := \frac{1}{|H|} \sum\limits_{h \in H} \phi_i\downarrow(h) \cdot \overline{\tau_j(h)}$$
 with the characters $\tau_j$ of $H$.
 By Frobenius reciprocity, the induced map of representation rings is given by the matrix $(\phi_i\downarrow | \tau_j)_{i,j}$.
When $H$ is the trivial group, this matrix is the one-column-matrix of values of the characters of $G$ on the neutral element.
For the non-trivial inclusions amongst finite subgroups of the Bianchi groups, let us compute this matrix case by case.
\begin{itemize}
\item Any inclusion $\Z/\ell \hookrightarrow \Sthree$ maps the generator of $\Z/\ell$
to the only conjugacy class of elements of order $\ell$ in $\Sthree$. 
So it induces the map obtained by restricting to two cycles,
\begin{itemize}
 \item for $\ell = 2$, the cycles $(1)$ and $(12)$: 
$$\begin{array}{c|rr|cc}
\Z/2 \hookrightarrow \Sthree & (1) & (12) & (\pi_i \downarrow | \rho_1 + \rho_2) & (\pi_i \downarrow | \rho_2) \\
\hline 
	    \pi_1 \downarrow & 1 & 1  & 1  & 0 \\ 
\widetilde{\pi_2} \downarrow & 0 & -2 & 0 & 1 \\
\widetilde{\pi_3} \downarrow & 0 & 0  & 0  & 0 \\
\end{array}$$
\item for $\ell = 3$,  the cycles $(1)$ and $(123)$:
$$\begin{array}{c|rr|ccc}
\Z/3 \hookrightarrow \Sthree & (1) & (123) & (\pi_i \downarrow | \sum_i \sigma_i) & (\pi_i \downarrow | \sigma_2) 
& (\pi_i \downarrow | \sigma_3)\\
\hline 
            \pi_1 \downarrow & 1 & 1  & 1  & 0 & 0\\ 
\widetilde{\pi_2} \downarrow & 0 & 0  & 0  & 0 & 0 \\
\widetilde{\pi_3} \downarrow & 0 & -3 & 0 & 1 & 1 \\
\end{array}$$
\end{itemize}
\item There are three possible inclusions $\Z/2 \hookrightarrow \Kleinfourgroup$,
 namely $g \mapsto a$, $g\mapsto b$ and $g\mapsto ab$.
 For each of these inclusions, we restrict to the image of $\Z/2$:
$$\begin{array}{c|rr|cc}
g \mapsto a & 1 & a & (\xi_i \downarrow | \rho_1 +\rho_2) & (\xi_i \downarrow | \rho_2) \\
\hline 
\xi_1 	       \downarrow & 1 & 1 & 1 & 0 \\ 
(\xi_2 -\xi_1) \downarrow & 0 & -2 & 0 & 1 \\
(\xi_3 -\xi_1) \downarrow & 0 & -2 & 0 & 1 \\
(\xi_4 -\xi_1) \downarrow & 0 & 0 & 0 & 0 \\
\end{array},
$$
$$
\begin{array}{c|rr|cc}
g \mapsto b & 1 & b & (\xi_i \downarrow | \rho_1 +\rho_2) & (\xi_i \downarrow | \rho_2) \\
\hline 
\xi_1 	       \downarrow & 1 & 1 & 1 & 0 \\ 
(\xi_2 -\xi_1) \downarrow & 0 & -2 & 0 & 1 \\
(\xi_3 -\xi_1) \downarrow & 0 & 0 & 0 & 0 \\
(\xi_4 -\xi_1) \downarrow & 0 & -2 & 0 & 1 \\
\end{array},
$$
$$
\begin{array}{c|rr|cc}
g \mapsto ab & 1 & ab & (\xi_i \downarrow | \rho_1 +\rho_2) & (\xi_i \downarrow | \rho_2) \\
\hline 
\xi_1	       \downarrow & 1 & 1 & 1 & 0 \\ 
(\xi_2 -\xi_1) \downarrow & 0 & 0 & 0 & 0 \\
(\xi_3 -\xi_1) \downarrow & 0 & -2 & 0 & 1 \\
(\xi_4 -\xi_1) \downarrow & 0 & -2 & 0 & 1 \\
\end{array}.$$
\item Any inclusion $\Z/2 \hookrightarrow \Afour$ maps the generator of $\Z/2$
to the only conjugacy class of elements of order 2 in $\Afour$.
So it induces the map obtained by restricting to the two cycles $(1)$ and $(12)(34)$.
\begin{center}
\begin{tabular}{c|rr|rr}
	$\Z/2 \hookrightarrow \Afour$& $1$ & $(12)(34)$ 
		& $(\chi_i \downarrow | \rho_1 +\rho_2)$ & $(\chi_i \downarrow | \rho_2)$ \\
       \hline
			       $\chi_1 \downarrow$ & $1$ & $1$ & $1$  & $0$\\
       $(\chi_2 -\chi_1 -\chi_3 -\chi_4) \downarrow$ & $0$ &$-4$ & $0$ & $2$\\
		       $(\chi_3 -\chi_1) \downarrow$ & $0$ & $0$ & $0$ & $0$\\
		       $(\chi_4 -\chi_1) \downarrow$ & $0$ & $0$ & $0$ & $0$
\end{tabular}\end{center}
An inclusion $\Z/3 \hookrightarrow \Afour$ can either map the generator $h$ of $\Z/3$ to the conjugacy class of $(123)$,
or to the conjugacy class of its square $(132)$. So we have the two possibilities
\begin{center}
\begin{tabular}{c|rrr|rrr}
        $h \mapsto (123)$& $1$ & $(123)$ & $(132)$ & $(\sum \chi_i \downarrow | \sum_i \sigma_i)$ 
		& $(\sum \chi_i \downarrow | \sigma_2)$  & $(\sum \chi_i \downarrow | \sigma_3)$\\
       \hline
				 $\chi_1 \downarrow$ & $1$ & $1$ & $1$      & $1$ & $0$ & $0$\\
       $(\chi_2 -\chi_1 -\chi_3 -\chi_4) \downarrow$ & $0$ & $0$ & $0$      & $0$&  $0$ & $0$\\
		       $(\chi_3 -\chi_1) \downarrow$ & $0$ & $j-1$& $j^2-1$ & $0$ & $1$ & $0$\\
		       $(\chi_4 -\chi_1) \downarrow$ & $0$ & $j^2-1$ &$j-1$ & $0$ & $0$& $1$
\end{tabular}\end{center}
and
\begin{center}
\begin{tabular}{c|rrr|rrr}
       $h \mapsto (132)$& $1$ & $(132)$ & $(123)$ & $(\sum \chi_i \downarrow | \sum_i \sigma_i)$ 
		& $(\sum \chi_i \downarrow | \sigma_2)$  & $(\sum \chi_i \downarrow | \sigma_3)$\\
       \hline
				 $\chi_1 \downarrow$ & $1$ & $1$ & $1$      & $1$ & $0$ & $0$\\
       $(\chi_2 -\chi_1 -\chi_3 -\chi_4) \downarrow$ & $0$ & $0$ & $0$      & $0$&  $0$ & $0$\\
		       $(\chi_3 -\chi_1) \downarrow$ & $0$ & $j^2-1$& $j-1$ & $0$ & $0$ & $1$\\
		       $(\chi_4 -\chi_1) \downarrow$ & $0$ & $j-1$ &$j^2-1$ & $0$ & $1$& $0$
\end{tabular}\end{center}
\end{itemize}
So we observe the claimed simultaneous diagonal block splitting.
\end{proof}

\begin{remark} \label{simultaneisation}
It has been pointed out in~\cite{Rahm_homological_torsion}*{observation 50}
that there are only two homeomorphism types of connected components which can occur in the orbit space of the 3-torsion subcomplex.
An observation which facilitates our Bredon homology computations can be made on the connected components
which are homeomorphic to a closed interval.
Let us make a consideration on the preimage of such a component. \\
Concerning our refined cell complex, by~\cite{Rahm_homological_torsion}*{lemma 16}
at each vertex $v$ of stabiliser type $\Afour$, we have two adjacent edges of stabiliser type $\Z/3$
modulo the action of the vertex stabiliser. 
% Given a copy of $\Afour$ in our Bianchi group stabilising a vertex $v$ on $C$,
We attribute the cycle $(123) \in \Afour$
to the image of the generator of one of the two stabilisers of type $\Z/3$ of representative edges adjacent to $v$.
Then we choose the preimage of $(123)$ under the inclusion of the other copy of $\Z/3$ to be the generator $h$.
Since we have assumed that $v$ maps to a point on a connected component homeomorphic to a closed interval in the quotient the 
$3$-torsion subcomplex, $v$ is the only point at which $h$ can be related to our first copy of $\Z/3$.
\\
So, we have obtained bases in which all the inclusions $\Z/3 \hookrightarrow \Afour$
induce the first of the two possibilities in the last item of the proof of Lemma~\ref{representation ring splitting}.
\end{remark}

\begin{proof}[Proof of Theorem $\ref{splitting}$]
By the properties of the cell complex $X$ described in Section~\ref{the cell complex},
all edge stabilisers are finite cyclic, all $2$-cell stabilisers are trivial and there are no higher-dimensional cells.
So the Bredon chain complex is concentrated in dimensions $0$, $1$ and $2$,
and has differentials
$$\xymatrix{
0 \ar[r] & \bigoplus\limits_{\sigma \in \thinspace_\Gamma \backslash X^{(2)}} R_\C (\{1\}) \ar[r]^{\Psi_2 } & 
\bigoplus\limits_{\sigma \in \thinspace_\Gamma \backslash X^{(1)}} R_\C (\Z/_{n\Z})  \ar[r]^{\Psi_1} &
\bigoplus\limits_{\sigma \in \thinspace_\Gamma \backslash X^{(0)}} R_\C (\Gamma_\sigma) \ar[r] & 0,
}  $$
where $\Psi_1$, $\Psi_2$ are concentrated in blocks given by the cell stabiliser inclusions.
Those blocks split by Lemma~\ref{representation ring splitting}, 
and therefore we can reorder the rows and columns of $\Psi_1$ 
such that $\Psi_1$ becomes concentrated in three blocks on the diagonal,
\begin{enumerate}
 \item $\Psi_1^{(1)}$, induced by the trivial and regular representations,
 \item $\Psi_1^{(2)}$, induced by the $2$--torsion subgroups and
 \item $\Psi_1^{(3)}$, induced by the $3$--torsion subgroups.
\end{enumerate}
Let us denote by $R_\C^{(1)}(G)$ the subring generated by the regular representation of $G$
when $G$ is a cyclic group, respectively by the trivial representation when 
$G$ is a finite group of different type.
The trivial representation and the regular representation do obviously coincide when $G$ is the trivial group.
As the blocks of $\Psi_2$ are induced exclusively by maps $\{1\} \to \Z/_{n\Z}$,
the differential $\Psi_2$ is concentrated exclusively in blocks induced by regular representations.
Therefore, from the Bredon chain complex we can split off a sequence
$$\xymatrix{
0 \ar[r] & \bigoplus\limits_{\sigma \in \thinspace_\Gamma \backslash X^{(2)}} R_\C (\{1\}) \ar[r]^{\Psi_2 } & 
\bigoplus\limits_{\sigma \in \thinspace_\Gamma \backslash X^{(1)}} R_\C^{(1)} (\Z/_{n\Z})  \ar[r]^{\Psi_1^{(1)}} &
\bigoplus\limits_{\sigma \in \thinspace_\Gamma \backslash X^{(0)}} R_\C^{(1)}  (\Gamma_\sigma) \ar[r] & 0,
}  $$
isomorphic to the cellular chain complex
$$\xymatrix{
0 \ar[r] & C_2({_\Gamma \backslash X}) \ar[r]^{\partial_2 } & 
C_1({_\Gamma \backslash X}) \ar[r]^{\partial_1 } & 
C_0({_\Gamma \backslash X}) \ar[r] &0,}  $$
computing the orbit space homology $\Homol_\bullet(\underbar{B}\Gamma ; \thinspace \Z)$.
The remainder of the Bredon chain complex is then concentrated in dimensions $0$ and $1$,
and splits into a direct summand for $\ell = 2$ and another one for $\ell = 3$,
both of the shape 
$$\xymatrix{
0 \ar[r] & \bigoplus\limits_{\sigma \in \thinspace_\Gamma \backslash X^{(1)}} R_\C^{(\ell)} (\Z/_{n\Z})  \ar[r]^{\Psi_1^{(\ell)}} &
\bigoplus\limits_{\sigma \in \thinspace_\Gamma \backslash X^{(0)}} R_\C^{(\ell)}  (\Gamma_\sigma) \ar[r] & 0,
}  $$
where $R_\C^{(\ell)}$ is defined by the source, respectively the target of $\Psi_1^{(\ell)}$ . 
\end{proof}

\section{Reduction of the torsion subcomplexes} \label{Reduction of the torsion subcomplexes}

\begin{proof}[Proof of Theorem {\rm \ref{2}}.]
 The part of the Bredon complex that is due to $2$-torsion is carried by the $2$-torsion subcomplex of the action of $\Gamma$ on hyperbolic space,
 and has the shape
$$0 \to \bigoplus_{\rm 1-cell \medspace orbits} R_\C^{(2)}(\Z/2) \aTop{\Psi_1^{(2)}}{\longrightarrow}  \bigoplus_{\rm 0-cell}^{\rm orbits} R_\C^{(2)}(\Z/2) \bigoplus_{\rm 0-cell}^{\rm orbits} R_\C^{(2)}(\Kleinfourgroup) \bigoplus_{\rm 0-cell}^{\rm orbits} R_\C^{(2)}(\Sthree) \bigoplus_{\rm 0-cell}^{\rm orbits} R_\C^{(2)}(\Afour) \to 0,$$
where the direct sums run over cells with the respective stabiliser.
At any vertex which is stabilised by a copy of $\Z/2$, we have two adjacent edges with the same stabiliser type,
and the cell stabiliser inclusions $\Z/2 \hookrightarrow \Z/2$ induce isomorphisms on representation rings.
So before splitting the representation rings, in the orbit space we can merge the two edges into one, 
getting rid of the vertex, without changing the homology of the Bredon complex.
The part of the latter that is due to $2$-torsion then has the partially reduced shape
$$0 \to \bigoplus_{\rm 1-cell \medspace orbits}^{\rm (less)} R_\C^{(2)}(\Z/2) \aTop{\Psi_1^{(2)}}{\longrightarrow} \bigoplus_{j = 1}^{o_2} R_\C^{(2)}(\Z/2) \bigoplus_{\rm 0-cell}^{\rm orbits} R_\C^{(2)}(\Kleinfourgroup) \bigoplus_{\rm 0-cell}^{\rm orbits} R_\C^{(2)}(\Sthree) \bigoplus_{\rm 0-cell}^{\rm orbits} R_\C^{(2)}(\Afour) \to 0,$$
where any of the vertices counted by $j$ is the only vertex in a connected component of homeomorphism type $\circlegraph$
in the quotient of the $2$-torsion subcomplex.
The splitting of Section~\ref{Representation ring splitting} implies that the matrix for $\Psi_1^{(2)}$ has blocks
\begin{itemize}
 \item \begin{flushright} $\begin{pmatrix}
        1 \\ 1 \\0
       \end{pmatrix}$
      ,
$\begin{pmatrix}
        0 \\ 1 \\ 1
       \end{pmatrix}$
       or
$\begin{pmatrix}
        1 \\0 \\ 1
       \end{pmatrix}$
       at inclusions $\Z/2 \hookrightarrow \Kleinfourgroup$, depending on which of the three cyclic subgroups of $\Kleinfourgroup$ is hit, 
\end{flushright}
 \item  $(1)$ at inclusions $\Z/2 \hookrightarrow \Sthree$,
 \item $(2)$
       at inclusions $\Z/2 \hookrightarrow \Afour$. 
       \end{itemize}
At any vertex which is stabilised by a copy of $\Sthree$, we have two adjacent edges with stabiliser type $\Z/2$,
and the cell stabiliser inclusions $\Z/2 \hookrightarrow \Sthree$ induce isomorphisms on the $2$-torsion parts of the splitted representation rings.
So we can merge the two edges into one, getting rid of the vertex, without changing the homology of the Bredon complex.
The latter then has the completely reduced shape
$$0 \to \bigoplus_{j=1}^{z_2} R_\C^{(2)}(\Z/2) \aTop{\Psi_1^{(2)}}{\longrightarrow} \bigoplus_{j = 1}^{o_2} R_\C^{(2)}(\Z/2) \bigoplus_{\rm 0-cell}^{\rm orbits} R_\C^{(2)}(\Kleinfourgroup) \bigoplus_{\rm 0-cell}^{\rm orbits} R_\C^{(2)}(\Afour) \to 0,$$
where the vertices of type $\Kleinfourgroup$ are the bifurcation points in the connected components of types $\graphFive$ and $\graphTwo$,
and the vertices of type $\Afour$ are the endpoints in the connected components of types $\graphTwo$ and $\edgegraph$.
Norbert Kr\"amer~\cite{Kraemer}*{Satz 8.3 and Satz 8.4} has shown that the types $\circlegraph$, $\graphFive$, $\graphTwo$ and $\edgegraph$ are all possible homeomorphism types which 
connected components of the $2$-torsion subcomplex of the action of a Bianchi group on hyperbolic space can have.

As the matrix blocks between distinct connected components are zero,
we now only need to compute the homology on a component of type $\circlegraph$,
and take it to the multiplicity $o_2$, 
and on components of types $\edgegraph$, $\graphFive$, $\graphTwo$, and figure out the multiplicities for the latter.
\begin{itemize}
 \item On a connected component of type $\circlegraph$, the map
 $ R_\C^{(2)}(\Z/2) \aTop{\Psi_1^{(2)}|_{\circlegraph}}{ \xrightarrow{\hspace*{1.2cm}} }   R_\C^{(2)}(\Z/2) $
 must be the zero map, because of the opposing signs at edge origin and edge end. Therefore,
 $\Homol_n\left(\Psi_\bullet^{(2)} |_{\circlegraph}\right) \cong \begin{cases}
                                      \Z,& n = 0 \medspace \text{\rm or }1,\\
                                      0,&\text{\rm otherwise}.
                                     \end{cases}
  $
 \item On a connected component of type $\edgegraph$ however, the map
 $$ R_\C^{(2)}(\Z/2) \aTop{\Psi_1^{(2)}|_{\edgegraph}}{ \xrightarrow{\hspace*{1.2cm}} }    R_\C^{(2)}(\Afour) \oplus  R_\C^{(2)}(\Afour) $$
 must be the diagonal map concatenated with multiplication by $2$ and alternating sign, 
 $\Z \to \Z \oplus \Z$,
 $x \mapsto (-2x, 2x)$,
 because of the matrix block $(2)$ at inclusions $\Z/2 \hookrightarrow \Afour$.
 This yields $\Homol_n\left(\Psi_\bullet^{(2)} |_{\edgegraph}\right) \cong \begin{cases}
                                      \Z \oplus \Z/2,& n = 0,\\
                                      0,&\text{\rm otherwise}.
                                     \end{cases}
  $
 \item On a connected component of type $\graphFive$, the three maps coming from the three cyclic subgroups in $\Kleinfourgroup$
do together constitute a matrix block
\begin{center} $ \Z^3 \cong \left(R_\C^{(2)}(\Z/2)\right)^3 
\aTop{\Psi_1^{(2)}|_{\graphFive}}{ \xrightarrow{\hspace*{1.2cm}} }  
\left( R_\C^{(2)}(\Kleinfourgroup) \right)^2 \cong \Z^6, 
\qquad  \Psi_1^{(2)}|_{\graphFive} =$ \scriptsize $\begin{pmatrix}
        -1& -1 &0 \\-1&0&-1 \\0&-1&-1 \\ 1& 1 &0 \\1&0&1 \\0&1&1 
       \end{pmatrix} $\normalsize
\end{center}
This matrix has elementary divisors $2$ of multiplicity one, and $1$ of multiplicity two.
 This yields $\Homol_n\left(\Psi_\bullet^{(2)} |_{\graphFive}\right) \cong \begin{cases}
                                      \Z^3 \oplus \Z/2,& n = 0,\\
                                      0,&\text{\rm otherwise}.
                                     \end{cases}
  $
 \item On a connected component of type $\graphTwo$, the two maps coming from the edge stabilisers
do together constitute a matrix block
\begin{center} $ \Z^2 \cong \left(R_\C^{(2)}(\Z/2)\right)^2 
\aTop{\Psi_1^{(2)}|_{\graphTwo}}{ \xrightarrow{\hspace*{1.2cm}} }  
 R_\C^{(2)}(\Kleinfourgroup) \oplus R_\C^{(2)}(\Afour) \cong \Z^4, 
\qquad  \Psi_1^{(2)}|_{\graphTwo} =$ \scriptsize $\begin{pmatrix}
        1-0 & 1\\1-1&0 \\0-1 &1 \\ 0 &-2 
       \end{pmatrix} $\normalsize
\end{center}
This matrix has elementary divisors $2$ and $1$, of multiplicity one each.
 This yields $\Homol_n\left(\Psi_\bullet^{(2)} |_{\graphTwo}\right) \cong \begin{cases}
                                      \Z^2 \oplus \Z/2,& n = 0,\\
                                      0,&\text{\rm otherwise}.
                                     \end{cases}
  $
\end{itemize}
Finally, we observe that precisely for each connected component except for those of type $\circlegraph$, 
we obtain one summand $\Z/2$ for $\Homol_0\left(\Psi_\bullet^{(2)}\right)$;
and that each such connected component admits two orbits of $\Kleinfourgroup$, 
whether contained in $\Afour$ or not
(these orbits are counted by $d_2$).
And precisely for each orbit of reduced edges (i.e., conjugacy classes of $\Z/2$ in $\Gamma$, counted by $z_2$), we have obtained a summand $\Z$ for $\Homol_0\left(\Psi_\bullet^{(2)}\right)$.
\end{proof}

\begin{proof}[Proof of Theorem {\rm \ref{3}}.]
 The part of the Bredon complex that is due to $3$-torsion is carried by the $3$-torsion subcomplex of the action of $\Gamma$ on hyperbolic space,
 and has the shape
$$0 \to \bigoplus_{\rm 1-cell \medspace orbits} R_\C^{(3)}(\Z/3) \aTop{\Psi_1^{(3)}}{\longrightarrow}  \bigoplus_{\rm 0-cell \medspace orbits} R_\C^{(3)}(\Z/3) \bigoplus_{\rm 0-cell \medspace orbits} R_\C^{(3)}(\Sthree) \bigoplus_{\rm 0-cell \medspace orbits} R_\C^{(3)}(\Afour) \to 0,$$
where the direct sums run over cells with the respective stabiliser.
At any vertex which is stabilised by a copy of $\Z/3$, we have two adjacent edges with the same stabiliser type,
and the cell stabiliser inclusions $\Z/3 \hookrightarrow \Z/3$ induce isomorphisms on representation rings.
So before splitting the representation rings, in the orbit space we can merge the two edges into one, 
getting rid of the vertex, without changing the homology of the Bredon complex.
The part of the latter that is due to $3$-torsion then has the partially reduced shape
$$0 \to \bigoplus_{\rm 1-cell \medspace orbits}^{\rm (less)} R_\C^{(3)}(\Z/3) \aTop{\Psi_1^{(3)}}{\longrightarrow} \bigoplus_{j = 1}^{o_3} R_\C^{(3)}(\Z/3)  \bigoplus_{\rm 0-cell \medspace orbits} R_\C^{(3)}(\Sthree) \bigoplus_{\rm 0-cell \medspace orbits} R_\C^{(3)}(\Afour) \to 0,$$
where any of the vertices counted by $j$ is the only vertex in a connected component of homeomorphism type $\circlegraph$
in the quotient of the $3$-torsion subcomplex.
The splitting of Section~\ref{Representation ring splitting} implies that the matrix for $\Psi_1^{(3)}$ has blocks
\begin{itemize}
 \item  $(1, 1)$ at inclusions $\Z/3 \hookrightarrow \Sthree$,
 \item $\begin{pmatrix}
        1 & 0 \\ 0 & 1
       \end{pmatrix}$
       or
$\begin{pmatrix}
        0 & 1 \\ 1 & 0
       \end{pmatrix}$
       at inclusions $\Z/3 \hookrightarrow \Afour$ 
       (by Remark~\ref{simultaneisation}, we can choose bases such that we exclusively get the first of these two blocks, if we desire),
\end{itemize}
At any vertex which is stabilised by a copy of $\Afour$, we have two adjacent edges with stabiliser type $\Z/3$,
and the cell stabiliser inclusions $\Z/3 \hookrightarrow \Afour$ induce isomorphisms on the $3$-torsion parts of the splitted representation rings.
So we can merge the two edges into one, getting rid of the vertex, without changing the homology of the Bredon complex.
The latter then has the completely reduced shape
$$0 \to \bigoplus_{j = 1}^{o_3 +\iota_3} R_\C^{(3)}(\Z/3) \aTop{\Psi_1^{(3)}}{\longrightarrow}    \bigoplus_{j = 1}^{o_3} R_\C^{(3)}(\Z/3) \bigoplus_{j = 1}^{2\iota_3} R_\C^{(3)}(\Sthree) \to 0,$$
where the vertices of type $\Sthree$ are the exclusive vertices of connected components of type $\edgegraph$.
As the matrix blocks between distinct connected components are zero, we now only need to compute the homology on a component of type $\circlegraph$,
and take it to the multiplicity $o_3$, and on a component of type $\edgegraph$,
and take it to the multiplicity $\iota_3$.
\begin{itemize}
 \item On a connected component of type $\circlegraph$, the map
 $ R_\C^{(3)}(\Z/3) \aTop{\Psi_1^{(3)}|_{\circlegraph}}{\longrightarrow}    R_\C^{(3)}(\Z/3) $
 must be the zero map, because of the opposing signs at edge origin and edge end. Therefore,
 $\Homol_n\left(\Psi_\bullet^{(3)} |_{\circlegraph}\right) \cong \begin{cases}
                                      \Z^{2},& n = 0 \medspace \text{\rm or }1,\\
                                      0,&\text{\rm otherwise}.
                                     \end{cases}
  $
 \item On a connected component of type $\edgegraph$ however, the map
 $$ R_\C^{(3)}(\Z/3) \aTop{\Psi_1^{(3)}|_{\edgegraph}}{\longrightarrow}    R_\C^{(3)}(\Sthree) \oplus  R_\C^{(3)}(\Sthree) $$
 must be the the map
 $\Z^2 \to \Z \oplus \Z$,
 $(x,y) \mapsto (-x-y, x+y)$,
 because of the matrix block $(1, 1)$ at inclusions $\Z/3 \hookrightarrow \Sthree$.
 This yields $\Homol_n\left(\Psi_\bullet^{(3)} |_{\edgegraph}\right) \cong \begin{cases}
                                      \Z,& n = 0 \medspace \text{\rm or }1,\\
                                      0,&\text{\rm otherwise}.
                                     \end{cases}
  $
\end{itemize}
Counting the connected components yields the claimed multiplicities.
\end{proof}

\section{Representation ring splitting in a more general setting} \label{general}

Representation ring splitting, as introduced in this paper, relies on the groups under study admitting a nice system of subgroups.
Without this prerequisite, we can still make the statements in the present section,
but this does not essentially improve on what has already been obtained by L\"uck and Oliver~\cite{LueckOliver} 
based on the split coefficient systems of S\l{}ominska~\cite{Slominska}.

First, we want to sharpen the following lemma of~\cite{MislinValette}.
This will allow us to apply representation ring splitting to determine the free part of the equivariant
 $K$-homology of any group with vanishing geometric torsion dimension, 
for instance the Hilbert modular groups and the Fuchsian groups.

\begin{lemma}[\cite{MislinValette}] \label{numberOfConjugacyClasses}
 Let $G$ be an arbitrary group and write ${\rm FC}(G)$ for the set of conjugacy classes of elements of finite order in $G$.
Then there is an isomorphism
$$\Homol^\mathfrak{Fin}_0(G; R_\C) \otimes_\Z \C \cong \C[{\rm FC}(G)].$$
\end{lemma}

Let $\underbar{\rm E}G^{\rm sing}$ denote the singular part of the classifying space
$\underbar{\rm E}G$ for proper actions of a group $G$, 
namely the subcomplex consisiting of all points in $\underbar{\rm E}G$ with non-trivial stabiliser. 
% (a nicer notation for it might be ``$\underbar{\rm E}G^{\rm tors}$'').
Mislin~\cite{MislinValette} indicates that up to $G$-homotopy, $\underbar{\rm E}G^{\rm sing}$ is uniquely determined by $G$.
Hence, $\dim \underbar{\rm E}G^{\rm sing}$ can be defined as the minimal dimension of $\underbar{\rm E}G^{\rm sing}$
within its $G$-homotopy type.

Mislin~\cite{MislinValette} shows that for any group $G$, there is a natural map
\begin{equation} \label{natural map}
  \Homol_q^\mathfrak{Fin}( \underbar{\rm E}G ; \thinspace R_\C) 
\to \Homol_q( \underbar{\rm B}G ; \thinspace \Z),
\end{equation}
which is an isomorphism in dimensions $q> \dim \underbar{\rm E}G^{\rm sing}+1$,
\\
and injective in dimension $q = \dim \underbar{\rm E}G^{\rm sing} +1$.
In the special case of vanishing geometric torsion dimension,
we can sharpen this statement together with Lemma~\ref{numberOfConjugacyClasses} as follows.

\begin{proposition}
 For any group $\Gamma$ with $\dim \underbar{\rm E}\Gamma^{\rm sing} = 0$, we have
$\Homol_q^\mathfrak{Fin}( \underbar{\rm E}\Gamma ; \thinspace R_\C) 
\cong \Homol_q( \underbar{\rm B}\Gamma ; \thinspace \Z)$
in degrees $q > 0$,
and $\Homol_0^\mathfrak{Fin}( \underbar{\rm E}\Gamma ; \thinspace R_\C) 
\cong \Z[{\rm FC}(\Gamma)]$.
\end{proposition}

\begin{proof}
 
Consider a classifying  space
$\underbar{\rm E}\Gamma$ with zero-dimensional $\underbar{\rm E}\Gamma^{\rm sing}$.
Here, the Bredon complex of $\underbar{\rm E}\Gamma$ is reduced in positive degrees to the cellular chain complex of
 $\underbar{\rm B}\Gamma$. 
For any stabiliser group of a vertex in $\underbar{\rm E}\Gamma^{\rm sing}$, with character table $\{\xi_1, \hdots, \xi_n\}$,
 where $\xi_1$ is the trivial character, we choose the basis
 $$\{\xi_1, \xi_2 -\xi_2(1)\xi_1,\hdots, \xi_n -\xi_n(1)\xi_1\}$$
 for its complex representation ring.
Then clearly any morphism induced by inclusions of the trivial edge stabilisers,
 induces a map with image contained in the submodule generated by $\xi_1$ in the representation ring.
So the Bredon complex splits in degree $0$ into a direct sum of the submodules generated by the bases
 $\{\xi_2 -\xi_2(1)\xi_1,\hdots, \xi_n -\xi_n(1)\xi_1\}$ at each vertex of $\underbar{\rm B}\Gamma^{\rm sing}$,
 and the image of the trivial characters of the edges.
The latter image is isomorphic to the image of the boundary map from $1$-cells to $0$-cells 
in~$\underbar{\rm B}\Gamma$,  so we obtain
\begin{itemize}
 \item a monomorphism $\Homol_1( \underbar{\rm B}\Gamma ; \Z) 
\hookrightarrow \Homol_1^\mathfrak{Fin}( \underbar{\rm E}\Gamma ; \thinspace R_\C) $

\medskip

\item and an isomorphism $\Homol_0^\mathfrak{Fin}( \underbar{\rm E}\Gamma ; \thinspace R_\C) 
\cong \Homol_0( \underbar{\rm B}\Gamma ; \Z)  \oplus \Z[{\rm FC}(\Gamma) \setminus \{1\}]$.
\end{itemize}
Using Mislin's map~(\ref{natural map}), we obtain the claimed result.
\end{proof}

For $\Gamma$ a Hilbert modular group or orientable Fuchsian group, $\dim \underbar{\rm E}\Gamma^{\rm sing} = 0$,
so we can apply the above proposition.

\section{Recalls on the assembly map} \label{The assembly map}
We will only outline the meaning of this assembly map.
For any discrete group $G$, we define its reduced $C^*$-algebra, $C^*_r(G)$, and the $K$-theory of the latter as in~\cite{Valette}.
According to Georges Skandalis, the algebras $C^*_r(G)$ are amongst the most important and natural examples of $C^*$-algebras.
It is difficult to compute their $K$-theory directly, 
so we use a homomorphism constructed by Paul Baum and Alain Connes~\cite{BaumConnesHigson}, 
$$   \mu_i : K_i^G(\underbar{E}G) \longrightarrow K_i(C_r^*(G)), \quad \quad
i \in \N \cup \{0\}, $$
called the (analytical) assembly map.
A model for the classifying space for proper actions,
written $\underbar{E}G$, is in the case of the Bianchi groups given by hyperbolic three-space
(being isomorphic to their associated symmetric space), 
the stabilisers of their action on it being finite.
% For the definition of its equivariant $K$-homology $K_i^G({\thinspace})$, we refer to~\cite{MislinValette}.
The Baum--Connes conjecture now states that the assembly map is an isomorphism.
The conjecture can be stated more generally, see~\cite{BaumConnesHigson}.

The Baum--Connes conjecture implies several important conjectures in topology, geometry, algebra and functional analysis.
Groups for which the assembly map is surjective verify the Kaplansky---Kadison conjecture on the idempotents.
Groups for which it is injective satisfy the strong Novikov conjecture and one sense of the Gromov---Lawson---Rosenberg conjecture,
namely the sense predicting the vanishing of the higher $\hat{A}$-genera
(see the book of Mislin and Valette~\cite{MislinValette} for details).

\subsection*{The assembly map for the Bianchi groups.} 
Julg and Kasparov~\cite{JulgKasparov} have shown that the Baum--Connes conjecture is verified for all the discrete subgroups of SO$(n,1)$
and SU$(n,1)$. The Lobachevski model of hyperbolic three-space gives a natural identification of its orientation preserving  isometries,
with matrices in $\PSO(3,1)$.
So especially, the assembly map is an isomorphism for the Bianchi groups;
and we have obtained the isomorphism type of $K_i(C_r^*(\Gamma))$.

An alternative way to check that the Baum--Connes conjecture is verified for the Bianchi groups,
is using ``a-T-menability'' in the sense of Gromov, 
also called the Haagerup property~\cite{CherixCowlingJolissaintJulgValette}.
Cherix, Martin and Valette prove in~\cite{CherixMartinValette} that among other groups,
the Bianchi groups admit a proper action on a space ``with measured walls''.
In~\cite{CherixCowlingJolissaintJulgValette}, Haglund, Paulin and Valette show that groups with such an action have the Haagerup property.
Finally Higson and Kasparov~\cite{HigsonKasparov} prove that the latter property implies the bijectivity of several assembly maps, 
and in particular the Baum--Connes conjecture.

The assembly map for the Bianchi groups is more deeply studied in~\cite{Fuchs}.

\section{Appendix} \label{the appendix}
On the machine~\cite{BianchiGP}, we obtain the following results:

\bigskip

\begin{tabular}{c|c|c|c|c|c|c}
$m$ &  \begin{tabular}{c} class \\ no. \end{tabular} &
  $\beta_1$ & 
\begin{tabular}{c} {2-torsion sub-} \\ {complex, reduced}  \end{tabular} &
\begin{tabular}{c} {3-torsion sub-} \\ {complex, reduced}  \end{tabular} &
$\Homol_0^\mathfrak{Fin}$ & $\Homol_1^\mathfrak{Fin}$
\\ 
\hline & & & & & &
 \\
% 1   &  1 & 0 & \graphOne & \edgegraph & $\Z^6$ & $\Z \oplus \Betti$ \\
% 3   &  1 & 0 & \edgegraph & \graphThree & $ \Z^5 \oplus \Z/2 $ & $ \Betti$  \\
7   &  1 & 1 & \circlegraph & \edgegraph & $ \Z^3 $ & $\Z^2 \oplus \Betti $ \\
2   &  1 & 1 & \graphTwo \qquad \quad & \circlegraph & $\Z^5 \oplus \Z/2 $ & $ \Z^2 \oplus \Betti $ \\
11  &  1 & 1 & \edgegraph & \circlegraph & $ \Z^4 \oplus \Z/2 $ & $ \Z^2 \oplus \Betti$ \\
19  &  1 & 1 & \edgegraph & \edgegraph & $ \Z^3 \oplus \Z/2 $ & $ \Z \oplus \Betti $ \\
15  &  2 & 2 & \circlegraph & \circlegraph & $ \Z^4 $ & $ \Z^3 \oplus \Betti $ \\
5   &  2 & 2 & \graphFive & \circlegraph & $ \Z^6 \oplus \Z/2 $ & $ \Z^2 \oplus \Betti $ \\
6   &  2 & 2 & \circlegraph \edgegraph & \circlegraph & $ \Z^5 \oplus \Z/2 $ & $ \Z^3 \oplus \Betti $ \\
43  &  1 & 2 & \edgegraph & \edgegraph & $ \Z^3 \oplus \Z/2 $ & $ \Z \oplus \Betti $\\
35  &  2 & 3 & \circlegraph & \circlegraph & $\Z^4 $ & $ \Z^3 \oplus \Betti$ \\
10  &  2 & 3 & \graphFive & \circlegraph & $ \Z^6 \oplus \Z/2 $ & $ \Z^2 \oplus \Betti $ \\
51  &  2 & 3 & \edgegraph \edgegraph & \circlegraph & $ \Z^5 \oplus \Z/2 \oplus \Z/2 $ & $ \Z^2 \oplus \Betti$ \\
13  &  2 & 3 & \graphFive & \edgegraph \edgegraph & $\Z^6 \oplus \Z/2 $ & $\Z^2 \oplus \Betti $ \\
67  &  1 & 3 & \edgegraph & \edgegraph & $ \Z^3 \oplus \Z/2 $ & $ \Z \oplus \Betti $ \\
22  &  2 & 5 &  \circlegraph \edgegraph & \circlegraph & $\Z^5 \oplus \Z/2 $ & $\Z^3 \oplus \Betti $ \\
91  &  2 & 5 & \circlegraph & \edgegraph \edgegraph & $\Z^4 $ & $\Z^3 \oplus \Betti $ \\
115 &  2 & 7  & \circlegraph & \circlegraph & $ \Z^4 $ & $ \Z^3 \oplus \Betti $ \\
123 & 2 & 7  & \edgegraph \edgegraph & \circlegraph & $ \Z^5 \oplus \Z/2 \oplus \Z/2 $ & $ \Z^2 \oplus \Betti$ \\
163 & 1 & 7 & \edgegraph & \edgegraph & $ \Z^3 \oplus \Z/2 $ & $ \Z \oplus \Betti $ \\
37  & 2 & 8  & \graphFive \circlegraph \circlegraph & \edgegraph \edgegraph & $ \Z^8 \oplus \Z/2 $ & $ \Z^4 \oplus \Betti $ \\
187 & 2 & 9 & \edgegraph \edgegraph & \circlegraph  & $ \Z^5 \oplus \Z/2 \oplus \Z/2 $ & $ \Z^2 \oplus \Betti$ \\
58  & 2 & 12  & \graphFive & \circlegraph & $\Z^6 \oplus \Z/2 $ & $ \Z^2 \oplus \Betti$ \\
235 & 2 & 13  &  \circlegraph \circlegraph \circlegraph & \circlegraph & $ \Z^6 $ & $ \Z^5 \oplus \Betti$ \\
267 & 2 & 15 & \edgegraph \edgegraph & \circlegraph & $\Z^5 \oplus \Z/2 \oplus \Z/2 $ & $ \Z^2 \oplus \Betti$ \\
403 & 2 & 19  & \circlegraph & \edgegraph \edgegraph & $ \Z^4 $ & $ \Z^3 \oplus \Betti$ \\
427 & 2 & 21  & \circlegraph \circlegraph \circlegraph & \edgegraph \edgegraph & $ \Z^6 $ & $ \Z^5 \oplus \Betti$ \\
\end{tabular}

\end{document}